\documentclass[11pt,letterpaper]{article}
\usepackage[margin=0.9in]{geometry}
\usepackage[T1]{fontenc}
\usepackage{lmodern,microtype}
\usepackage{authblk}
\usepackage{mathtools,amssymb,amsthm}
\usepackage{booktabs,array}
\usepackage{tikz}
\usepackage{float}
\tikzset{
  ex A/.style={circle,draw=black!85,fill=black!85,text=white,
    minimum size=6mm,inner sep=1pt,font=\small},
  ex B/.style={ex A,fill=white,text=black!85},
  ex edge/.style={draw=black!30,line width=.55pt},
  ex matching/.style={draw=teal!65!black,line width=1pt},
  ex text/.style={anchor=west,font=\small,inner sep=0pt}
}
\usetikzlibrary{arrows.meta,decorations.markings}
\usepackage[hidelinks]{hyperref}
\hypersetup{pdftitle={Resolving Pisanski's Conjecture on Quadrilateral Embeddings},pdfauthor={Alexander Metzger and Austin Ulrigg}}
\newtheorem{theorem}{Theorem}[section]
\newtheorem{lemma}[theorem]{Lemma}
\newtheorem{proposition}[theorem]{Proposition}

\theoremstyle{definition}
\newtheorem{definition}[theorem]{Definition}

\theoremstyle{remark}

\newcommand{\Z}{\mathbb Z}

\newcommand{\ga}{\gamma}
\newcommand{\Dih}{\mathrm D}
\title{Resolving Pisanski's Conjecture on Quadrilateral Embeddings}\author[1]{Alexander Metzger}
\author[2]{Austin Ulrigg}
\affil[1]{University of California, Berkeley,
\texttt{alexmetzger@berkeley.edu}}
\affil[2]{University of Washington,
\texttt{austinul@uw.edu}}
\date{}

\begin{document}
\maketitle

\begin{abstract}
We prove Pisanski's conjecture on quadrilateral embeddings of regular bipartite graphs. For integers $m\geq d\geq4$ with $4\mid md$, we construct both a connected simple $d$-regular bipartite graph $G$ with $m$ vertices in each bipartition class and an orientable quadrangulation. This graph's genus is
\[
\ga(G)=1+\frac{m(d-4)}4.
\]
These quadrangulations attain the lower bound on the genus from Euler's formula. We give
dihedral constructions for even $m$ and a cyclic construction
for every $m\geq d$ with $4\mid d$. In each case, two cyclic
permutations of a connection set of size $d$ determine all of the
local rotations.
\end{abstract}

\section{Introduction}\label{sec:introduction}

In 1980, Toma\v{z} Pisanski conjectured the following for positive integers $m$ and $d$. Whenever $m\geq d\geq4$ and $4\mid md$, there exists a connected $d$-regular bipartite graph, with $m$ vertices in each class, that has an orientable quadrangulation~\cite[Conjecture~1]{Pisanski1980}. We prove this conjecture by constructing an orientable quadrangulation for every such pair $(m,d)$.

For a graph $G$, let $\ga(G)$ denote its orientable genus. If a $d$-regular bipartite graph with $m$ vertices in each class has an orientable quadrangulation, the surface has genus $1+m(d-4)/4$, attaining the lower bound on the graph's genus from Euler's formula~\cite[p.~33]{Pisanski1980}. Having $d$ distinct neighbors in the opposite class also requires $m\geq d$. The surface's genus must be an integer, so $4\mid md$ is necessary. Pisanski's conjecture asks whether these necessary conditions are sufficient when $d\geq4$.

Many of the graph families whose orientable genus is known are highly symmetric. In 1965, Ringel determined the genus of the complete bipartite graphs~\cite{Ringel1965}. In particular, $K_{m,m}$ has an orientable quadrangulation for every even $m$. Mohar, Parsons and Pisanski determined the orientable genus of nearly complete bipartite graphs~\cite{MoharParsonsPisanski1985}. For a perfect matching $M$ in $K_{m,m}$, they proved
\[
\ga(K_{m,m}-M)=\left\lceil\frac{(m-1)(m-4)}4\right\rceil
\qquad(m\geq3).
\]
When $m\equiv0,1\pmod4$, the expression inside the ceiling is an integer and the bound is attained by a quadrangulation. For $m\geq5$, these give cases of Pisanski's conjecture with $d=m-1$. Alpert and Gross~\cite{AlpertGross1976} and Anderson~\cite{Anderson1981} proved further partial cases. Anderson proved existence when $m$ and $d$ are both even, and when $m$ is odd with $d\equiv4\pmod8$ and $d\geq12$. More recently, Ulrigg and Metzger~\cite{MetzgerUlrigg2026} determined the genus of an infinite family of bipartite Kneser graphs, giving explicit quadrangulations that realize the bound in Pisanski's conjecture. 

The constructions of this paper exploit group-invariant rotation systems, where we choose the group and connection set together with the local rotations. These freedoms allow us to ensure that every facial walk closes after exactly $4$ steps while ensuring that both permutations are cyclic.

\begin{theorem}[Main theorem]\label{thm:main}
Let $m$ and $d$ be integers with $m\geq d\geq4$. There exists a connected simple $d$-regular bipartite graph with $m$ vertices in each bipartition class and a cellular orientable embedding whose face boundaries are simple $4$-cycles if and only if $4\mid md$. Whenever $4\mid md$, the graph $G$ constructed below has minimum orientable genus
\[
\ga(G)=1+\frac{m(d-4)}4.
\]
\end{theorem}

We construct Haar graphs~\cite{EstelyiPisanski2016} of groups of order $m$, using two cyclic orders on a
connection set of size $d$ to specify all local rotations.
For even $m$ and $d$, this recovers the
graph family quadrangulated by Anderson%
~\cite[pp.~357--358]{Anderson1981}. The dihedral construction
also covers every admissible pair with odd $d$. Our cyclic construction applies whenever $4\mid d$, extending
the degree range for the cyclic graph family in Anderson's
Corollary~2.2 to include $8\mid d$. Together these constructions cover every pair $(m,d)$ satisfying the conditions of Pisanski's conjecture.

Sections~\ref{sec:preliminaries}--\ref{sec:group} give the preliminary definitions and background material, and Sections~\ref{sec:dihedral}--\ref{sec:completion} give the constructions and proof.

\section{Preliminaries}\label{sec:preliminaries}

We follow the definitions of \emph{Graphs on Surfaces} by Mohar and Thomassen~\cite{MoharThomassen2001}. Throughout this paper, graphs are finite, simple, and undirected, and surfaces are connected and closed.

An embedding of a graph on a surface represents its vertices by distinct points and its edges by simple arcs joining their endpoints. These arcs meet only at common endpoints and do not pass through any other vertices. The connected components of the complement of the embedded graph are its \emph{faces}. An embedding is \emph{cellular}, or a \emph{2-cell embedding}, if every face is homeomorphic to an open disk. The \emph{orientable genus} $\ga(G)$ of a graph $G$ is the smallest nonnegative integer $g$ for which $G$ embeds in the orientable surface $S_g$, the connected sum of $g$ tori, where $S_0$ is the sphere. It is known that every minimum-genus embedding of a connected graph in an orientable surface is cellular~\cite{MoharThomassen2001,Youngs1963}.

An \emph{orientable quadrangulation} is a cellular embedding in an
orientable surface in which every facial walk has length four.
If $G$ is connected, simple, bipartite, and $d$-regular, with $m$
vertices in each bipartition class and $d\geq4$, Euler's formula
gives the lower bound on the orientable genus of $G$~\cite[p.~33]{Pisanski1980}.
\begin{equation}\label{eq:euler-bound}
\ga(G)\geq\left\lceil1+\frac{m(d-4)}4\right\rceil.
\end{equation}
If all faces of an orientable embedding of $G$ are quadrilaterals, as in an orientable quadrangulation, the surface on which $G$ embeds has genus $1+m(d-4)/4$, attaining this bound and having $4\mid md$.

Cellular embeddings can be described combinatorially by listing the
local rotations at each vertex. Given a cellular embedding of $G$
in an orientable surface $S$, fix an orientation of $S$. Around
each vertex $v$, we can list its incident edges in clockwise order.
This defines a cyclic permutation $\pi_v$, where $\pi_v(e)$ is the
edge that follows $e$ in that order. Each permutation $\pi_v$ is
called the \emph{local rotation} at $v$, and the set $\Pi=\{\pi_v\mid v\in V(G)\}$ is the \emph{rotation system} of the embedding. More generally,
a choice of one cyclic permutation of the incident edges at each
vertex is called a rotation system of $G$. The Heffter--Edmonds--Ringel rotation principle gives a one-to-one
correspondence between rotation systems and cellular orientable embeddings~\cite[pp.~87--90]{MoharThomassen2001}.

\begin{theorem}[Heffter--Edmonds--Ringel]
\label{thm:rotation-principle}
Every rotation system of a connected graph with at least one edge
determines a cellular embedding in an orientable surface.
This embedding is unique up to a homeomorphism that induces the
identity on the graph.
\end{theorem}

We now describe the facial walks determined by a given rotation system.
A \emph{dart} of $G$ is a pair $(v,e)$, where $e$ is an edge
incident with $v$. If $e$ joins $v$ and $w$, its two darts are
$(v,e)$ and $(w,e)$. Let $R$ denote vertex rotation and $\tau$
denote dart reversal, so
\[
R(v,e)=(v,\pi_v(e)),
\qquad
\tau(v,e)=(w,e).
\]
We compose permutations from right to left. The \emph{facial permutation} is
\begin{equation}\label{eq:face-convention}
F=R\tau.
\end{equation}
Repeatedly applying $F$ returns to the starting dart and traces a closed
walk, called a \emph{facial walk}. The orbits of $F$ are the
facial walks of the embedding, up to cyclic shift. Therefore, a rotation system of a connected graph whose facial walks all have length four determines an orientable quadrangulation.

\section{Haar graphs and Quadrangulations}\label{sec:group}

Representing the vertices of $G$ by elements of a finite group, with one copy for each bipartition class, we choose one cyclic order for the connection set for all $A$-vertices and another for all $B$-vertices.

\begin{definition}[Haar graph]\label{def:haar}
Let $\Gamma$ be a finite group and $S\subseteq\Gamma$. Define $H(\Gamma,S)$ with vertices
\[
\{A_g:g\in\Gamma\}\ \cup\ \{B_g:g\in\Gamma\}
\]
and edges $\{A_gB_{gs}:g\in\Gamma,\ s\in S\}$.
\end{definition}

The two displayed vertex sets are disjoint, even when their group subscripts agree. In particular, the identity may belong to $S$. The graph $H(\Gamma,S)$ is simple, bipartite, and $|S|$-regular, with $|\Gamma|$ vertices in each class. For nonempty $S$, it is connected if and only if $SS^{-1}=\{st^{-1}:s,t\in S\}$ generates $\Gamma$~\cite[Lemma~1(i)]{EstelyiPisanski2016}. In particular, if $1_\Gamma\in S$ and $S$ generates $\Gamma$, then $H(\Gamma,S)$ is connected. Let $|S|=d$, and let $\rho_A$ and $\rho_B$ be $d$-cycles on $S$. Write $(g,s)_A$ for the dart $A_g\to B_{gs}$ and $(g,s)_B$ for the dart $B_g\to A_{gs^{-1}}$. Define the rotations by
\[
R(g,s)_A=(g,\rho_A(s))_A,
\qquad
R(g,s)_B=(g,\rho_B(s))_B.
\]
At every vertex, the label determines the neighboring vertex uniquely, so these are cyclic orders of all incident darts. Dart reversal is given by
\[
\tau(g,s)_A=(gs,s)_B,
\qquad
\tau(g,s)_B=(gs^{-1},s)_A.
\]
Thus, the face permutation satisfies
\begin{align}
F(g,s)_A&=(gs,\rho_B(s))_B,\label{eq:first-step}\\
F(g,s)_B&=(gs^{-1},\rho_A(s))_A,\notag
\intertext{and hence}
F^2(g,s)_A&=\bigl(gs\rho_B(s)^{-1},\rho_A\rho_B(s)\bigr)_A.
\label{eq:two-step}
\end{align}

In~\eqref{eq:two-step}, the label changes by $\rho_A\rho_B$ and the group subscript is multiplied on the right by $s\rho_B(s)^{-1}$. We choose $\rho_A$ and $\rho_B$ so that applying this rule twice returns to the starting dart, but applying it once does not. Every face contains an $A$-dart, so this gives all facial orbits. In the dihedral constructions, we take $\rho_A=\rho_B^{-1}$ and choose the cyclic order $\rho_B$ so that $s\rho_B(s)^{-1}$ has order two for every $s\in S$. By~\eqref{eq:two-step}, two facial steps leave the label fixed and change the group subscript. Two further steps then return to the original dart.

\begin{figure}[ht]
\centering
\scalebox{0.8}{%
\begin{tikzpicture}[
    font=\small,
    walk/.style args={#1}{
        draw=#1, line width=.85pt,
        -{Stealth[length=5.4pt,width=4pt]},
        shorten <=4.6pt, shorten >=4.6pt,
        preaction={draw=#1!13,line width=5pt,-}
    },
    group step/.style args={#1}{
        circle,fill=#1!10,draw=#1!25,line width=.4pt,
        minimum size=18pt,inner sep=2pt,
        text=#1,font=\small\boldmath
    },
    vertex/.style={circle,draw=facialdart,line width=.8pt,
        minimum size=6pt,inner sep=0pt},
    vertex text/.style={inner sep=0pt},
    dart/.style={inner sep=0pt}
]
\definecolor{facialdart}{HTML}{2F5364}
\colorlet{stepone}{facialdart}
\colorlet{steptwo}{teal!65!black}
\colorlet{stepthree}{orange!65!black}
\colorlet{stepfour}{violet!65!black}

\coordinate (v0) at (135:2.1);
\coordinate (v1) at (45:2.1);
\coordinate (v2) at (-45:2.1);
\coordinate (v3) at (-135:2.1);

\draw[walk=stepone] (v0)
    arc[start angle=135,end angle=45,radius=2.1];
\draw[walk=steptwo] (v1)
    arc[start angle=45,end angle=-45,radius=2.1];
\draw[walk=stepthree] (v2)
    arc[start angle=-45,end angle=-135,radius=2.1];
\draw[walk=stepfour] (v3)
    arc[start angle=225,end angle=135,radius=2.1];

\node[group step=stepone] at (0,1.55) {$s_i$};
\node[group step=steptwo] at (1.55,0) {$s_{i+1}^{-1}$};
\node[group step=stepthree] at (0,-1.55) {$s_i$};
\node[group step=stepfour] at (-1.55,0) {$s_{i+1}^{-1}$};

\node[dart] at (0,2.65) {$(g,s_i)_A$};
\node[dart,anchor=west] at (2.65,0)
    {$(gs_i,s_{i+1})_B$};
\node[dart] at (0,-2.65) {$(gq_i,s_i)_A$};
\node[dart,anchor=east] at (-2.65,0)
    {$(gs_{i+1},s_{i+1})_B$};

\node[vertex,fill=facialdart] at (v0) {};
\node[vertex,fill=white] at (v1) {};
\node[vertex,fill=facialdart] at (v2) {};
\node[vertex,fill=white] at (v3) {};

\node[vertex text,anchor=south east]
    at ([xshift=-6pt,yshift=6pt]v0) {$A_g$};
\node[vertex text,anchor=south west]
    at ([xshift=6pt,yshift=6pt]v1) {$B_{gs_i}$};
\node[vertex text,anchor=north west]
    at ([xshift=6pt,yshift=-6pt]v2) {$A_{gq_i}$};
\node[vertex text,anchor=north east]
    at ([xshift=-6pt,yshift=-6pt]v3) {$B_{gs_{i+1}}$};
\end{tikzpicture}%
}
\caption{{A face of Lemma~\ref{lem:involutions}.
Each arrow is labeled by its dart and its group multiplier~\cite{MetzgerUlrigg2026}.}}
\label{fig:face}
\end{figure}
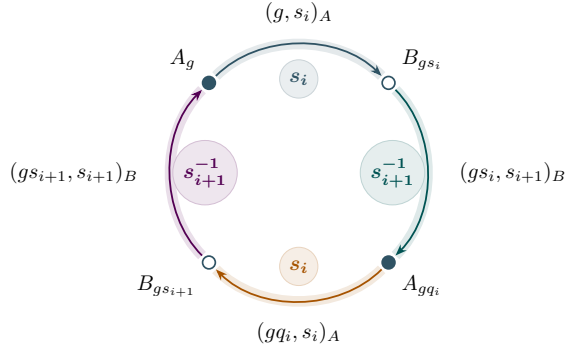

\begin{lemma}[Quadrangulations from involutions]\label{lem:involutions}
Let $s_0,\ldots,s_{d-1}$ be distinct elements of a finite group $\Gamma$, and put $S=\{s_0,\ldots,s_{d-1}\}$. Suppose that
\[
q_i=s_is_{i+1}^{-1}
\]
is a nonidentity involution for every $i$ modulo $d$. For connected $H(\Gamma,S)$, the rotations
\[
\rho_B=(s_0\ s_1\ \cdots\ s_{d-1}),\qquad \rho_A=\rho_B^{-1}
\]
give an orientable quadrangulation of $H(\Gamma,S)$. Every face boundary is a simple $4$-cycle.
\end{lemma}

\begin{proof}
Both local label permutations are $d$-cycles. By~\eqref{eq:two-step},
\[
F^2(g,s_i)_A=(gq_i,s_i)_A,
\qquad
F^4(g,s_i)_A=(gq_i^2,s_i)_A=(g,s_i)_A.
\]
The first equality does not return the original dart because $q_i\neq1_\Gamma$. Since $F$ interchanges the two types of darts, every facial orbit has length exactly four. Since $q_i^2=1_\Gamma$ and $q_i=s_is_{i+1}^{-1}$, we have $q_is_i=s_{i+1}$. The face starting with $(g,s_i)_A$ has vertices
\begin{equation}\label{eq:involution-face}
A_g,\quad B_{gs_i},\quad A_{gq_i},\quad B_{gs_{i+1}}.
\end{equation}
The two $A$-vertices are distinct because $q_i\neq1_\Gamma$. The two $B$-vertices are distinct because $s_i\neq s_{i+1}$. Vertices from different classes are distinct by definition.
\end{proof}

\section{The dihedral construction}\label{sec:dihedral}

Suppose $m=2n$ is even. We use the dihedral group of order $2n$, denoted by
\[
\Dih_n=\langle r,t\mid r^n=t^2=1,\ trt=r^{-1}\rangle.
\]
Its elements are $1,r,\ldots,r^{n-1}$ and $t,rt,\ldots,r^{n-1}t$. We call these the rotations and reflections, respectively. The relation $tr^a=r^{-a}t$ shows that every reflection is an involution. Further, when one of $x$ and $y$ is a rotation and the other is a reflection, the element $xy^{-1}$ is a reflection:
\begin{equation}\label{eq:dihedral-quotients}
r^a(r^bt)^{-1}=r^{a+b}t,
\qquad
(r^at)(r^b)^{-1}=r^{a+b}t.
\end{equation}
Thus, a cyclic list of elements of $\Dih_n$ that alternates between rotations and reflections satisfies the order-two condition in Lemma~\ref{lem:involutions}.

\begin{lemma}[Alternating rotations and reflections]
\label{lem:dihedral-alternation}
Let $d\geq2$ be even, and let
$s_0,\ldots,s_{d-1}\in\Dih_n$ alternate between rotations
and reflections. For every $i$ modulo $d$, the element
\[
q_i=s_is_{i+1}^{-1}
\]
satisfies $q_i\neq1$ and $q_i^2=1$.
\end{lemma}

\begin{proof}
Since $d$ is even, every consecutive pair consists of
one rotation and one reflection, including the last
and first entries. By~\eqref{eq:dihedral-quotients},
$q_i$ is therefore a reflection, say $q_i=r^at$.
A reflection is not the identity, and
\[
q_i^2=(r^at)^2
=r^a(tr^a)t
=r^ar^{-a}t^2
=1.
\]
\end{proof}

\subsection{Dihedral rotations for even \texorpdfstring{$m$ and $d$}{m and d}}

Anderson gave an orientable quadrangulation for this graph family%
~\cite[pp.~357--358]{Anderson1981} using White's composition
theorem~\cite{White1972}. We give explicit dihedral rotations
for the same graphs.

\begin{proposition}\label{prop:even-degree}
Let $m=2n$, $d=2k$ with $2\leq k\leq n$, and $S=\{r^a,r^at:0\leq a<k\}$. Then
\begin{equation}\label{eq:even-list}
\rho_B=(1\ t\ r\ rt\ r^2\ r^2t\ \cdots\ r^{k-1}\ r^{k-1}t),
\qquad \rho_A=\rho_B^{-1}
\end{equation}
give an orientable quadrangulation of the connected simple
$d$-regular bipartite graph $H(\Dih_n,S)$, with $m$ vertices
in each class and simple $4$-cycles as face boundaries.
For $k=2$, the cyclic list in~\eqref{eq:even-list} stops at $rt$.
\end{proposition}

\begin{proof}
The entries of $\rho_B$ are distinct and alternate cyclically between rotations and reflections. By Lemma~\ref{lem:dihedral-alternation}, every product of a label and the inverse of its successor has order two. Since $1,r,t\in S$, the graph is connected. Lemma~\ref{lem:involutions} therefore gives the required quadrangulation.
\end{proof}

When $d=m$, this is $K_{m,m}$ with genus $(m-2)^2/4$, recovering Ringel's formula for even $m$~\cite{Ringel1965}.
Figure~\ref{fig:example-even} shows the construction for
$m=6$ and $d=4$.

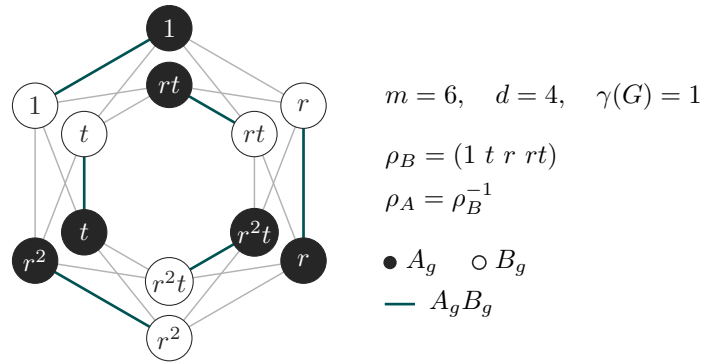
\begin{figure}[htbp]
\centering
\begin{tikzpicture}
\foreach \i in {0,...,5}{
  \coordinate (o\i) at ({90-60*\i}:2.05);
  \coordinate (i\i) at ({90-60*\i}:1.30);
}
\foreach \i in {0,...,5}{
  \pgfmathtruncatemacro{\j}{mod(\i+1,6)}
  \foreach \p in {o,i}{
    \foreach \q in {o,i}{
      \draw[ex edge] (\p\i)--(\q\j);
    }
  }
}
\foreach \a/\b in {o0/o5,i0/i1,o2/o1,i2/i3,o4/o3,i4/i5}{
  \draw[ex matching] (\a)--(\b);
}
\foreach \i/\a/\b/\c in {
  0/1/rt/A,1/r/rt/B,2/r/{r^2t}/A,
  3/{r^2}/{r^2t}/B,4/{r^2}/t/A,5/1/t/B}{
  \node[ex \c] at (o\i) {$\a$};
  \node[ex \c] at (i\i) {$\b$};
}
\node[ex text] at (2.85,1.15)
  {$m=6,\quad d=4,\quad \gamma(G)=1$};
\node[ex text] at (2.85,.35)
  {$\rho_B=(1\ t\ r\ rt)$};
\node[ex text] at (2.85,-.20)
  {$\rho_A=\rho_B^{-1}$};

\fill[black!85] (2.92,-1.05) circle (2.5pt);
\node[ex text] at (3.12,-1.05) {$A_g$};
\draw[draw=black!85,fill=white,line width=.6pt]
  (4.10,-1.05) circle (2.5pt);
\node[ex text] at (4.30,-1.05) {$B_g$};
\draw[ex matching] (2.85,-1.60)--(3.25,-1.60);
\node[ex text] at (3.43,-1.60) {$A_gB_g$};
\end{tikzpicture}
\caption{The even-degree dihedral group example
$H(\mathrm{D}_3,\{1,t,r,rt\})$.}
\label{fig:example-even}
\end{figure}

\subsection{Dihedral rotations for odd \texorpdfstring{$d$}{d}
and \texorpdfstring{$4\mid m$}{m divisible by four}}

Some of the odd-degree cases were established by Alpert and
Gross~\cite{AlpertGross1976}, and the case $d=m-1$ follows
from the genus of nearly complete bipartite
graphs~\cite{MoharParsonsPisanski1985}. The following construction
applies to every odd $d\geq5$ and every $m\geq d$ divisible
by four.

Suppose $d=2k+1$. An odd-length cyclic order cannot
alternate between rotations and reflections throughout.
We use $k+1$ rotations and $k$ reflections, arranged so
that exactly one consecutive pair consists of two rotations.
Since $d$ is odd, the condition $4\mid md$ requires
$4\mid m$, so $n=m/2$ is even. The rotation $z=r^{n/2}$
has order two, since $z\neq1$ and $z^2=r^n=1$.
We take $1$ and $z$ as the consecutive rotations,
giving $1z^{-1}=z$. Every other consecutive pair will
contain one rotation and one reflection.

\begin{proposition}\label{prop:odd-degree}
Let $m=2n$ with $4\mid m$ and
$5\leq d=2k+1\leq m$. Choose distinct rotations
$u_0,\ldots,u_k$ such that
\[
u_0=1,\qquad u_1=z=r^{n/2},\qquad u_2=r.
\]
Put $v_j=r^jt$ for $0\leq j<k$, and set
\[
S=\{u_0,\ldots,u_k\}\cup\{v_0,\ldots,v_{k-1}\}.
\]
Define
\begin{equation}\label{eq:odd-list}
\rho_B=(u_0\ u_1\ v_0\ u_2\ v_1\ \cdots\ u_k\ v_{k-1}),
\qquad
\rho_A=\rho_B^{-1}.
\end{equation}
These rotations give an orientable quadrangulation of
the connected simple $d$-regular bipartite graph
$H(\Dih_n,S)$, with $m$ vertices in each class and
simple $4$-cycles as face boundaries.
\end{proposition}

\begin{proof}
Since $d$ is odd and $d\leq2n$, we have
$2\leq k\leq n-1$. Also $n\geq4$, because
$4\mid m$ and $m\geq d\geq5$. Thus $1,z,r$ are
distinct, and $k+1\leq n$ leaves enough rotations
to choose the remaining $u_i$. The $k$ reflections
are distinct as well, so $S$ has exactly $d$ elements. Write $\rho_B=(s_0\ s_1\ \cdots\ s_{d-1})$.
The first consecutive product is
\[
s_0s_1^{-1}=u_0u_1^{-1}=z,
\]
which has order two. Every other consecutive pair
contains one rotation and one reflection, including
the last and first entries. Applying
Lemma~\ref{lem:dihedral-alternation} to each of these
two-element lists shows that $s_is_{i+1}^{-1}$ has
order two for every remaining index $i$.
Thus, for every $i$ modulo $d$,
\[
s_is_{i+1}^{-1}\neq1,
\qquad
\bigl(s_is_{i+1}^{-1}\bigr)^2=1.
\]
Since $1,r,t\in S$, the graph is connected.
Lemma~\ref{lem:involutions} therefore gives the
required quadrangulation.
\end{proof}

Figure~\ref{fig:example-odd} shows the construction for
$m=8$ and $d=5$, with $z=r^2$.

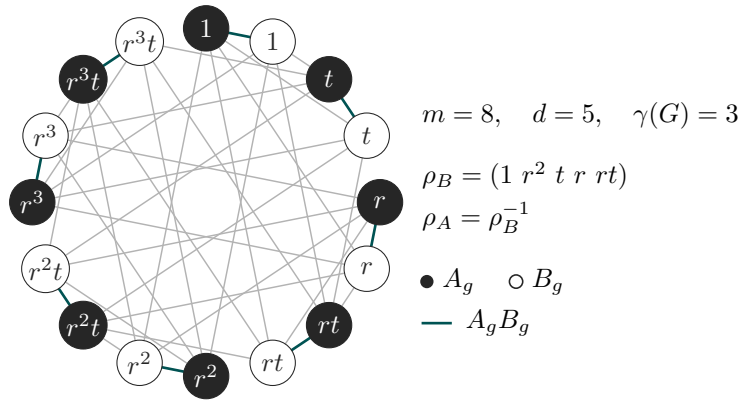
\begin{figure}[htbp]
\centering
\begin{tikzpicture}
\foreach \j in {0,...,7}{
  \coordinate (a\j) at ({90-45*\j}:2.30);
  \coordinate (b\j) at ({67.5-45*\j}:2.30);
}
\foreach \i in {0,...,3}{
  \foreach \e in {0,1}{
    \pgfmathtruncatemacro{\a}{2*\i+\e}
    \foreach \p/\q in {0/0,2/0,0/1,1/0,1/1}{
      \pgfmathtruncatemacro{\j}{mod(\i+(1-2*\e)*\p+4,4)}
      \pgfmathtruncatemacro{\f}{mod(\e+\q,2)}
      \pgfmathtruncatemacro{\b}{2*\j+\f}
      \draw[ex edge] (a\a)--(b\b);
    }
  }
}
\foreach \j in {0,...,7}{
  \draw[ex matching] (a\j)--(b\j);
}
\foreach \j/\g in {
  0/1,1/t,2/r,3/rt,4/{r^2},5/{r^2t},6/{r^3},7/{r^3t}}{
  \node[ex A] at (a\j) {$\g$};
  \node[ex B] at (b\j) {$\g$};
}
\node[ex text] at (2.85,1.15)
  {$m=8,\quad d=5,\quad \gamma(G)=3$};
\node[ex text] at (2.85,.35)
  {$\rho_B=(1\ r^2\ t\ r\ rt)$};
\node[ex text] at (2.85,-.20)
  {$\rho_A=\rho_B^{-1}$};

\fill[black!85] (2.92,-1.05) circle (2.5pt);
\node[ex text] at (3.12,-1.05) {$A_g$};
\draw[draw=black!85,fill=white,line width=.6pt]
  (4.10,-1.05) circle (2.5pt);
\node[ex text] at (4.30,-1.05) {$B_g$};
\draw[ex matching] (2.85,-1.60)--(3.25,-1.60);
\node[ex text] at (3.43,-1.60) {$A_gB_g$};
\end{tikzpicture}
\caption{The odd-degree dihedral group example
$H(\mathrm{D}_4,\{1,r^2,t,r,rt\})$.}
\label{fig:example-odd}
\end{figure}

\section{Cyclic rotations for degrees divisible by four}
\label{sec:cyclic}

The dihedral construction covers every admissible pair $(m,d)$ with
even $m$. For odd $m$, the condition $4\mid md$ requires
$4\mid d$. We give cyclic rotations that work for every
$m\geq d\geq4$ with $4\mid d$. For $m>d$, Anderson quadrangulated the same cyclic graph
family when $d\equiv4\pmod8$ and $d\geq12$%
~\cite[Corollary~2.2]{Anderson1981}. His graph is isomorphic to
$H(\Z_m,\{0,\ldots,d-1\})$. The rotations below also apply when $8\mid d$, completing the odd $m$ case. For $d=4$,
the graph is the circulant $C_{2m}(1,3)$, and our rotations recover its known toroidal quadrangulation~\cite[Proposition~5]{ConderGrande2015}.

Write $d=2h$, where $h$ is even, and using additive notation
in $\Z_m$, set
\begin{equation}\label{eq:cyclic-labels}
S=\{0,1,\ldots,2h-1\},
\qquad
\rho_B=(0\ 1\ \cdots\ h-1\ 2h-1\ 2h-2\ \cdots\ h),
\qquad
\rho_A=\rho_B^{h-1}.
\end{equation}
The labels in $S$ are distinct modulo $m$. Since $\rho_B$ is a $2h$-cycle and $h$ is even, $\gcd(2h,h-1)=\gcd(2,h-1)=1$. For a $k$-cycle $\pi$, it is standard that $\pi^N$ has exactly $\gcd(k,N)$ cycles each of length $k/\gcd(k,N)$. Thus our previous computation shows $\rho_A$ is also a
$2h$-cycle. The first half of the $\rho_B$-cycle increases from $0$ to $h-1$, and the second decreases from $2h-1$ to $h$. Thus labels $h$ positions apart sum to $2h-1$. Put $\sigma(s)=2h-1-s$. Then
\[
\rho_B^h=\sigma,\qquad \rho_A\rho_B=\sigma,\qquad \rho_B \sigma=\sigma\rho_B.
\]
The permutation $\sigma$ has no fixed points, since $2s$ is an even integer it cannot equal $2h-1$.

\begin{proposition}\label{prop:cyclic}
Let $m\geq d\geq4$ with $4\mid d$. The rotations
in~\eqref{eq:cyclic-labels} give an orientable quadrangulation
of $H(\Z_m,S)$, with $m$ vertices in each class, degree $d$,
and simple $4$-cycles as face boundaries.
\end{proposition}

\begin{proof}
The graph is simple, bipartite, and $d$-regular. It is
connected because $S$ contains the identity $0$ and the
generator $1$ of $\Z_m$. Both local rotations are $d$-cycles. In additive notation,~\eqref{eq:two-step} gives
\begin{equation}\label{eq:cyclic-two-step}
F^2(x,s)_A=\bigl(x+s-\rho_B(s),\sigma(s)\bigr)_A.
\end{equation}
Applying this twice yields
\[
F^4(x,s)_A=
\bigl(x+s-\rho_B(s)+\sigma(s)-\rho_B(\sigma(s)),s\bigr)_A.
\]
Since $\sigma$ commutes with $\rho_B$,
\[F^4(x,s)_A=(x+\underbrace{s+\sigma(s)}_{2h-1}-\underbrace{(\rho_B(s)+\sigma(\rho_B(s)))}_{2h-1},s)_A.\]
Thus $F^4(x,s)_A=(x,s)_A$. Since $F$ interchanges $A$-darts and $B$-darts and $\sigma$ has no fixed points, every
facial orbit has length exactly four. The face starting with $(x,s)_A$ has vertices
\begin{equation}\label{eq:cyclic-face}
A_x,\quad B_{x+s},\quad A_{x+s-\rho_B(s)},\quad
B_{x+(2h-1)-\rho_B(s)}.
\end{equation}
The two $A$-vertices are distinct because $s$ and
$\rho_B(s)$ are distinct modulo $m$. If the two $B$-vertices
were equal, then $\rho_B(s)\equiv \sigma(s)\pmod m$. Both labels
belong to $S$, so this would imply
$\rho_B(s)=\sigma(s)=\rho_B^h(s)$. Hence $\rho_B^{h-1}(s)=s$,
which is impossible because $\rho_B$ is a $2h$-cycle and
$0<h-1<2h$.
\end{proof}

For $m=7$ and $d=4$, the graph is shown in
Figure~\ref{fig:example-cyclic}.

\begin{figure}[H]
\centering
\begin{tikzpicture}
\foreach \i in {0,...,13}{
  \coordinate (v\i) at ({90-360*\i/14}:2.05);
}
\foreach \i in {0,...,13}{
  \pgfmathsetmacro{\a}{90-360*\i/14}
  \draw[ex edge] (v\i)
    arc[start angle=\a,delta angle={-360/14},radius=2.05];
  \pgfmathtruncatemacro{\j}{mod(\i+3,14)}
  \draw[ex edge] (v\i)--(v\j);
}
\foreach \x in {0,...,6}{
  \pgfmathtruncatemacro{\i}{2*\x}
  \pgfmathtruncatemacro{\j}{mod(2*\x+11,14)}
  \draw[ex matching] (v\i)--(v\j);
}
\foreach \i in {0,...,13}{
  \ifodd\i
    \pgfmathtruncatemacro{\g}{mod((\i+3)/2,7)}
    \node[ex B] at (v\i) {$\g$};
  \else
    \pgfmathtruncatemacro{\g}{\i/2}
    \node[ex A] at (v\i) {$\g$};
  \fi
}
\node[ex text] at (2.85,1.15)
  {$m=7,\quad d=4,\quad \gamma(G)=1$};
\node[ex text] at (2.85,.35)
  {$\rho_A=\rho_B=(0\ 1\ 3\ 2)$};
\node[ex text] at (2.85,-.20)
  {$\sigma(s)=3-s$};

\fill[black!85] (2.92,-1.05) circle (2.5pt);
\node[ex text] at (3.12,-1.05) {$A_g$};
\draw[draw=black!85,fill=white,line width=.6pt]
  (4.10,-1.05) circle (2.5pt);
\node[ex text] at (4.30,-1.05) {$B_g$};
\draw[ex matching] (2.85,-1.60)--(3.25,-1.60);
\node[ex text] at (3.43,-1.60) {$A_gB_g$};
\end{tikzpicture}
\caption{The cyclic group example
$H(\mathbb Z_7,\{0,1,2,3\})\cong C_{14}(1,3)$.}
\label{fig:example-cyclic}
\end{figure}

\section{Proof of the main theorem}\label{sec:completion}

Table~\ref{tab:cases} describes the three cases of the construction.

\begin{table}[H]
\centering
\begin{minipage}{0.86\linewidth}
\setlength{\abovecaptionskip}{0pt}
\setlength{\belowcaptionskip}{7pt}
\caption{Constructions for $m\geq d\geq4$ with $4\mid md$.}
\label{tab:cases}
\small
\renewcommand{\arraystretch}{1.25}
\begin{tabular*}{\linewidth}{@{\extracolsep{\fill}}cccl@{}}
\toprule
\textbf{Class size $m$} & \textbf{Degree $d$} & \textbf{Group} & \textbf{Construction}\\
\midrule
$2\mid m$  & $2\mid d$  & $\Dih_{m/2}$ & Proposition~\ref{prop:even-degree}\\
\addlinespace[2pt]
$4\mid m$  & $2\nmid d$ & $\Dih_{m/2}$ & Proposition~\ref{prop:odd-degree}\\
\addlinespace[2pt]
$2\nmid m$ & $4\mid d$  & $\Z_m$        & Proposition~\ref{prop:cyclic}\\
\bottomrule
\end{tabular*}
\end{minipage}
\end{table}

\begin{proof}[Proof of Theorem~\ref{thm:main}]
The necessity of $4\mid md$ follows from Euler's formula. Now suppose $4\mid md$. If $m$ and $d$ are even, use Proposition~\ref{prop:even-degree}. If $m$ is even and $d$ is odd, divisibility forces $4\mid m$, so Proposition~\ref{prop:odd-degree} applies. If $m$ is odd, then $4\mid d$, and Proposition~\ref{prop:cyclic} applies. Each construction gives a connected simple $d$-regular bipartite graph with $m$ vertices in each class and an orientable quadrangulation with simple $4$-cycles as face boundaries. Its vertex, edge, and face counts are $v=2m$, $e=md$, and $f=md/2$. Thus the surface on which $G$ embeds has genus $1+m(d-4)/4$, attaining the lower bound in~\eqref{eq:euler-bound}.
\end{proof}

\section{Further directions}\label{sec:directions}

Theorem~\ref{thm:main} settles the existence question when the graph may be chosen. However, for a given graph, the necessary conditions need not be sufficient. Pisanski demonstrated a connected $4$-regular bipartite graph with nine vertices in each class and girth four that has no orientable quadrangulation~\cite[p.~34]{Pisanski1980}. One may ask what structural conditions guarantee an orientable quadrangulation?

Moreover for $m\geq d\geq4$ with $4\nmid md$, one may seek an embedding attaining the rounded Euler bound $g_0=\lceil1+m(d-4)/4\rceil$. Writing $|f|$ for the length of a facial walk $f$, Euler's formula gives
\[
\sum_f\bigl(|f|-4\bigr)
=8\left(g_0-1-\frac{m(d-4)}4\right).
\]
Here the sum has one term for each face. For $md\equiv1,2,3\pmod4$, the sum is respectively $6,4,2$. Since every facial walk in a bipartite graph has even length, the last case requires one hexagon, the middle case two hexagons or one octagon, and the first case three hexagons, a hexagon and an octagon, or one decagon. All other faces must be quadrilaterals. Under what conditions can the previously given rotation systems be extended to obtain these face distributions?

For a nonorientable quadrangulation, Euler's formula gives genus $k=2+m(d-4)/2$, so $2\mid md$ is necessary. However, it is also insufficient: for $(m,d)=(5,4)$, the only graph is $K_{5,5}$ minus a perfect matching, which has no nonorientable quadrangulation~\cite[Theorem~3]{Mohar1988}, although Proposition~\ref{prop:cyclic} gives an orientable one. The nonorientable genus of nearly complete bipartite graphs is known~\cite{Mohar1988,SinghSun2023}. An extension of Pisanski's conjecture could ask which pairs $m\geq d\geq4$ admit a regular bipartite nonorientable quadrangulation.

We note that the freedom to choose the connection set is useful in this problem. We can choose its elements together
with the local rotations so that the facial walks have length four and both local permutations are cyclic. In the cyclic construction, the same label permutations work for every $m\geq d$. This motivates the question of whether similar choices give minimum-genus embeddings for
other families of Haar graphs.

On the other hand, for any fixed Haar graph the connection set is fixed,
and satisfying the conditions that make every facial walk return to its starting dart after exactly $4$ steps, while keeping each local permutation as a single cycle, can be much more difficult. It would be useful
to find structural conditions on a connection set that
guarantee suitable cyclic permutations. One can also ask which changes to the label permutations preserve
facial closure while combining their cycles. Such conditions or operations could extend these
constructions to prescribed graph families whose genus is unknown.

\section*{Acknowledgements}
This material is based upon work supported by the National Science Foundation Graduate
Research Fellowship Program under Grant No. DGE-2637800. Any opinions, findings,
and conclusions or recommendations expressed in this material are those of the author(s)
and do not necessarily reflect the views of the National Science Foundation.

\end{document}